\documentclass[11pt,a4paper]{amsart}
\usepackage{graphicx,multirow,array,amsmath,amssymb}

\newtheorem{theorem}{Theorem}

\newtheorem{proposition}[theorem]{Proposition}

\begin{document}

\title{Quantum knots and the number of knot mosaics}

\author[S. Oh]{Seungsang Oh}
\address{Department of Mathematics, Korea University, Anam-dong, Sungbuk-ku, Seoul 136-701, Korea}
\email{seungsang@korea.ac.kr}
\author[K. Hong]{Kyungpyo Hong}
\address{Department of Mathematics, Korea University, Anam-dong, Sungbuk-ku, Seoul 136-701, Korea}
\email{cguyhbjm@korea.ac.kr}
\author[H. Lee]{Ho Lee}
\address{Department of Mathematical Sciences, KAIST, 291 Daehak-ro, Yuseong-gu, Daejeon 305-701, Korea}
\email{figure8@kaist.ac.kr}
\author[H. J. Lee]{Hwa Jeong Lee}
\address{Department of Mathematical Sciences, KAIST, 291 Daehak-ro, Yuseong-gu, Daejeon 305-701, Korea}
\email{hjwith@kaist.ac.kr}

\thanks{2010 Mathematics Subject Classification: 57M25, 57M27, 81P15, 81P68}
\thanks{The corresponding author(Seungsang Oh) was supported by Basic Science Research Program through
the National Research Foundation of Korea(NRF) funded by the Ministry of Science,
ICT \& Future Planning(MSIP) (No.~2011-0021795).}
\thanks{This work was supported by the National Research Foundation of Korea(NRF) grant
funded by the Korea government(MEST) (No. 2011-0027989).}

\begin{abstract}
Lomonaco and Kauffman developed a knot mosaic system to introduce
a precise and workable definition of a quantum knot system.
This definition is intended to represent an actual physical quantum system.
A knot $(m,n)$-mosaic is an $m \times n$ matrix of mosaic tiles
($T_0$ through $T_{10}$ depicted in the introduction)
representing a knot or a link by adjoining properly that is called suitably connected.
$D^{(m,n)}$ is the total number of all knot $(m,n)$-mosaics.
This value indicates the dimension of the Hilbert space of these quantum knot system.
$D^{(m,n)}$ is already found for $m,n \leq 6$ by the authors.

In this paper, we construct an algorithm producing the precise value of $D^{(m,n)}$ for $m,n \geq 2$
that uses recurrence relations of state matrices
that turn out to be remarkably efficient to count knot mosaics.
$$ D^{(m,n)} = 2 \, \| (X_{m-2}+O_{m-2})^{n-2} \| $$
where $2^{m-2} \times 2^{m-2}$ matrices $X_{m-2}$ and $O_{m-2}$ are defined by
$$ X_{k+1} = 
\begin{bmatrix} X_k & O_k \\ O_k & X_k  \end{bmatrix}
\ \mbox{and } \
O_{k+1} = 
\begin{bmatrix} O_k & X_k \\ X_k & 4 \, O_k  \end{bmatrix} $$
for $k=0,1, \cdots, m-3$, with $1 \times 1$ matrices
$X_0 = \begin{bmatrix} 1  \end{bmatrix}$ and
$O_0 = \begin{bmatrix} 1  \end{bmatrix}$.
Here $\|N\|$ denotes the sum of all entries of a matrix $N$.
For $n=2$, $(X_{m-2}+O_{m-2})^0$ means the identity matrix of size $2^{m-2} \times 2^{m-2}$.
\end{abstract}

\maketitle

\section{Introduction} \label{sec:intro}
During the last three decades, much of the theory of knots has been applied in quantum physics.
One of remarkable discovery in the theory of knots is the Jones polynomial,
and it turned out that the explanation of the Jones polynomial has to do with quantum theory
\cite{J1, J2, K1, K2, L, LK2, SJ}.
Lomonaco and Kauffman introduced quantum knots to explain how to make quantum information versions
of mathematical structures in the series of papers \cite{LK1, LK3, LK4, LK5}.
They build a knot mosaic system to set the foundation for a quantum knot system,
which can be viewed as a blueprint for the construction of an actual physical quantum system.
Their definition of quantum knots was based on the planar projections of knots and the Reidemeister moves.
They model the topological information in a knot by a state vector in a Hilbert space
that is directly constructed from knot mosaics.

This paper is a sequel to the research program on finding the total number of knot mosaics
announced in \cite{HLLO1, HLLO2}.
This counting is very important because
the total number of knot mosaics is indeed the dimension of the Hilbert space of these quantum knot mosaics.
In \cite{HLLO2}, a partition matrix argument was developed by the authors to count small knot mosaics.
In this sequel, we generalize this argument to give an algorithm for counting all knot mosaics
that uses recurrence relations of matrices that are called state matrices.
This algorithm using state matrices turns out to be remarkably efficient to count knot mosaics.

Throughout this paper, the term ``knot" means either a knot or a link.
We begin by explaining the basic notion of knot mosaics.
Let $\mathbb{T}$ denote the set of the following eleven symbols that are called {\em mosaic tiles\/}; \\

\begin{figure}[h]
\includegraphics{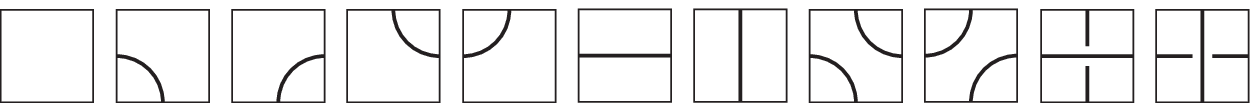}
\[ \ T_0 \hspace{8mm} T_1 \hspace{8mm} T_2 \hspace{8mm} T_3 \hspace{8mm} T_4 \hspace{8mm}
T_5 \hspace{8mm} T_6 \hspace{8mm} T_7 \hspace{8mm} T_8 \hspace{8mm}
T_9 \hspace{8mm} T_{10} \]
\vspace{-7mm}
\label{fi1}
\end{figure}

For positive integers $m$ and $n$,
an {\em $(m,n)$-mosaic\/} is an $m \times n$ matrix $M=(M_{ij})$ of mosaic tiles.
We denote the set of all $(m,n)$-mosaics by $\mathbb{M}^{(m,n)}$.
Note that $\mathbb{M}^{(m,n)}$ has $11^{mn}$ elements.
This definition is an extended version of the definition of an {\em $n$-mosaic\/}
as an $n \times n$ matrix of mosaic tiles in \cite{LK3}.

A {\em connection point\/} of a mosaic tile is defined as the midpoint of a mosaic tile edge
that is also the endpoint of a curve drawn on the tile.
Then each tile has zero, two or four connection points as follows; 

\begin{figure}[h]
\includegraphics{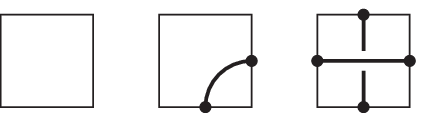}
\vspace{-2mm}
\label{fi2}
\end{figure}

We say that two tiles in a mosaic are {\em contiguous\/} if they lie immediately next to each other
in either the same row or the same column.
A mosaic is said to be {\em suitably connected\/}
if any pair of contiguous mosaic tiles have or do not have connection points simultaneously
on their common edge.
A {\em knot $(m,n)$-mosaic\/} is a suitably connected $(m,n)$-mosaic
whose boundary edges do not have connection points.
Then this knot $(m,n)$-mosaic represents a specific knot.
$\mathbb{K}^{(m,n)}$ denotes the subset of $\mathbb{M}^{(m,n)}$ of all knot $(m,n)$-mosaics.
A knot $(n,n)$-mosaic is simply specified by a {\em knot $n$-mosaic\/}.
The examples of mosaics in Figure \ref{fig1} are a non-knot $(4,5)$-mosaic and the trefoil knot 4-mosaic.
Also the reader finds a complete list of all 22 knot 3-mosaics in Appendix A in \cite{LK3}.

\begin{figure}[h]
\includegraphics{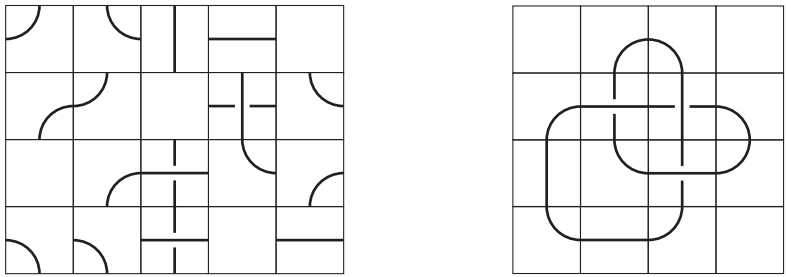}
\caption{Examples of mosaics}
\label{fig1}
\end{figure}

As an analog to the planar isotopy moves and the Reidemeister moves for standard knot diagrams,
Lomonaco and Kauffman \cite{LK3} created the $11$ mosaic planar isotopy moves
and the mosaic Reidemeister moves on knot mosaics.
They conjectured that
for any two tame knots $K_1$ and $K_2$,
and their arbitrary chosen mosaic representatives $M_1$ and $M_2$, respectively,
$K_1$ and $K_2$ are of the same knot type if and only if
$M_1$ and $M_2$ are of the same knot mosaic type, which is defined in \cite{LK3}.
This means that knot mosaic type is a complete invariant of tame knots.
Kuriya and Shehab \cite{KS} verified Lomonaco--Kauffman conjecture.

Lomonaco and Kauffman also proposed several questions related to knot mosaics.
$D^{(m,n)}$ denotes the total number of elements of $\mathbb{K}^{(m,n)}$.
In the series of recent papers \cite{HLLO1, HLLO2},
the authors found some results about $D^{(m,n)}$.
They showed that $D^{(1,n)}=1$, $D^{(2,n)}=2^{n-1}$ and $D^{(3,n)}=\frac{2}{5}(9 \cdot 6^{n-2} + 1)$
for a positive integer $n$,
and found a table of the precise values of $D^{(m,n)}$ for $m,n = 4,5,6$
(note that $D^{(m,n)} = D^{(n,m)}$);
\vspace{3mm}

\begin{center}
\begin{tabular}{|c|r|r|r|}   \hline
$D^{(m,n)}$ & $n=4$ & $n=5$ & $n=6$ \\    \hline
$m=4$ & $2594$ & $54,226$ & $1,144,526$ \\    \hline
$m=5$ & & $4,183,954$ & $331,745,962$ \\    \hline
$m=6$ & & & $101,393,411,126$ \\   \hline
\end{tabular}
\end{center}

\vspace{3mm}
\noindent Furthermore a lower and an upper bounds on $D^{(m,n)}$ for $m,n \geq 3$ were established;
$$2^{(m-3)(n-3)} \leq \frac{275}{2(9 \cdot 6^{m-2} + 1)(9 \cdot 6^{n-2} + 1)}
\cdot D^{(m,n)} \leq 4.4^{(m-3)(n-3)}.$$

In this paper, we construct an algorithm producing the precise value of $D^{(m,n)}$ in general.

\begin{theorem}\label{thm:main}
For integers $m,n \geq 2$,
the total number $D^{(m,n)}$ of all knot $(m,n)$-mosaics is the following;
$$ D^{(m,n)} = 2 \, \| (X_{m-2}+O_{m-2})^{n-2} \| $$
where $2^{m-2} \times 2^{m-2}$ matrices $X_{m-2}$ and $O_{m-2}$ are defined by
$$ X_{k+1} = 
\begin{bmatrix} X_k & O_k \\ O_k & X_k  \end{bmatrix}
\ \mbox{and } \
O_{k+1} = 
\begin{bmatrix} O_k & X_k \\ X_k & 4 \, O_k  \end{bmatrix} $$
for $k=0,1, \dots, m-3$, with $1 \times 1$ matrices
$X_0 = \begin{bmatrix} 1  \end{bmatrix}$ and
$O_0 = \begin{bmatrix} 1  \end{bmatrix}$.
\end{theorem}

Here $\| N \|$ denotes the sum of all entries of a matrix $N$.
For $n=2$, $(X_{m-2}+O_{m-2})^0$ means the identity matrix of size $2^{m-2} \times 2^{m-2}$.
We have calculated $D^{(n,n)}$ for $m=n=1,2, \dots, 13$
as given in the following table.
\vspace{3mm}

\begin{center}
{\scriptsize
\begin{tabular}{ll}      \hline \hline
$n$ \ \ \ & $D^{(n,n)}$ \\    \hline
1 & 1 \\
2 & 2 \\
3 & 22 \\
4 & 2594 \\
5 & 4183954 \\
6 & 101393411126 \\
7 & 38572794946976688 \\
8 & 234855052870954480828416 \\
9 & 23054099362200399656046175453184 \\
10 & 36564627559441092217310409777161751756800 \\
11 & 937273142571326423641676956468995920021677311787008 \\
12 & 388216021519370806221346434513102393133985590844312961759051776 \\
13 & 2597619491722287317211028202262384724016872304209163446959826047706385612800 \\   \hline \hline
\end{tabular}
}
\end{center}

\vspace{3mm}
Indeed $D^{(n,n)}$ grows in a quadratic exponential rate. 
The growth constant $\lim_{n \rightarrow \infty} (D^{(n,n)})^{\ \frac{1}{n^2}}$ exists and lies 
between  $4$ and $\frac{5+ \sqrt{13}}{2} \ (\approx 4.303)$.
This result was proved recently by Oh \cite{Oh}.
Other issues for knot mosaics involve considering  mosaic representations on the torus
rather than in the plane.
A knot toroidal $(m,n)$-mosaic is a suitably connected $(m,n)$-mosaic constructed on a torus
by identifying their boundaries properly.
Recently the authors and Yeon \cite{OHLLY} improved this state matrix algorithm to find the total number of
knot toroidal $(m,n)$-mosaics for positive co-prime integers $m$ and $n$.
Another result about knot toroidal mosaics is found in \cite{CL}.
Also a result about mosaic representations of graphs with at most 4 valencies
is considered in \cite{OH}.
Mirror-curve representations of knots that are similar to mosaic representations
were treated in \cite{JRSZ}.

Another interesting question related to knot mosaics is the mosaic number of knots.
Define the {\em mosaic number\/} $m(K)$ of a knot $K$ as the smallest integer $n$
for which $K$ is representable as a knot $n$-mosaic.
For example, the mosaic number of the trefoil is 4 as illustrated in Figure \ref{fig1}.
One question is the following:
{\em Is this mosaic number related to the crossing number of a knot?\/}
The authors \cite{LHLO} found an upper bound on the mosaic number as follows;
If $K$ is a nontrivial knot or a non-split link except the Hopf link,
then $m(K) \leq c(K) + 1$.
Moreover if $K$ is prime and non-alternating except the $6^3_3$ link, then $m(K) \leq c(K) - 1$.
Note that the mosaic numbers of the Hopf link and the $6^3_3$ link are 4 and 6, respectively.

\section{Suitably connected mosaics and state matrices}

Let $p$ and $q$ be positive integers.
$\mathbb{S}^{(p,q)}$ denotes the set of all suitably connected $(p,q)$-mosaics
that possibly have connection points on their boundary edges.
A suitably connected (5,3)-mosaic is depicted in Figure \ref{fig2} as an example.
This is a submosaic of a knot mosaic in Lomonaco and Kauffman's definition.

\begin{figure}[h]
\includegraphics{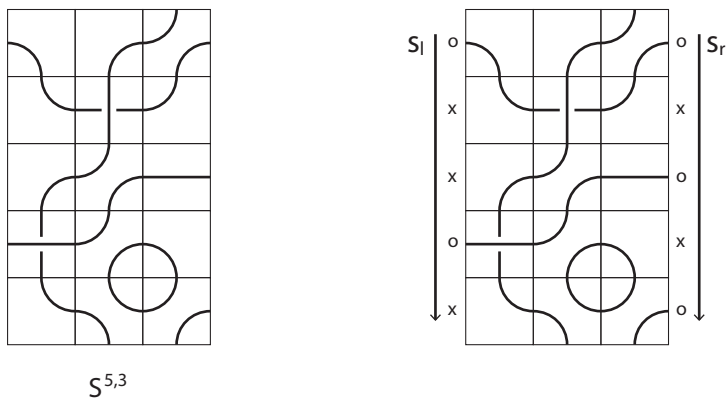}
\caption{Suitably connected (5,3)-mosaic $S^{5,3}$}
\label{fig2}
\end{figure}

For simplicity of exposition, a mosaic tile is called {\em $l$-, $r$-, $t$-\/} and {\em $b$-cp\/}
if it has a connection point on its left, right, top and bottom, respectively.
Sometimes we use two or more letters such as $lt$-cp for the case of both $l$-cp and $t$-cp.
Also we use the sign $\hat{}$ \/ for negation such as
$\hat{l}$-cp means not $l$-cp,
$\hat{l} \hat{t}$-cp means both $\hat{l}$-cp and $\hat{t}$-cp, and
$\widehat{lt}$-cp (which is differ from $\hat{l} \hat{t}$-cp) means the negation of $lt$-cp,
i.e., $\hat{l} t$-, $l \hat{t}$- or $\hat{l} \hat{t}$-cp.
\vspace{3mm}

\noindent {\bf Choice rule.\/} Each $M_{ij}$ in a suitably connected mosaic has four
choices of mosaic tiles as $T_7$, $T_8$, $T_9$ and $T_{10}$ if it is
$lrb$-cp (so automatically $t$-cp), and unique choice if it is $\widehat{lrb}$-cp.
\vspace{3mm}

For a suitably connected $(p,q)$-mosaic $S^{p,q} =  (M_{ij})$ where $i=1,\dots,p$ and $j=1,\dots,q$,
an {\em $l$-state\/} of $S^{p,q}$ indicates the presence of connection points
of $p$ mosaic tile edges on the leftmost boundary,
and we denote that $s_l(S^{p,q}) = s_l(M_{11}) s_l(M_{21}) \dots s_l(M_{p1})$
where each $s_l(M_{i1})$ denotes ``x'' if $M_{i1}$ is $\hat{l}$-cp and ``o'' if $M_{i1}$ is $l$-cp.
Similarly we define an {\em $r$-state\/} of $S^{p,q}$ which indicates the presence of connection points
of $p$ mosaic tile edges on the rightmost boundary.
For $(5,3)$-mosaic $S^{5,3}$ drawn in the figure,
$$s_l(S^{5,3}) = \mbox{oxxox \ and } s_r(S^{5,3}) = \mbox{oxoxo.}$$
Note that $\mathbb{S}^{(p,q)}$ has possibly $2^p$ kinds of $l$-states and also $2^p$ kinds of $r$-states.
We arrange the elements of the set of all states in the backward of lexicographical order
such as xxx, oxx, xox, oox, xxo, oxo, xoo and ooo for $p=3$.

Now we are ready to define a state matrix which turns out to be remarkably efficient to count
the number of suitably connected mosaics.
A {\em state matrix\/} for $\mathbb{S}^{(p,q)}$  is a $2^p \times 2^p$ matrix $N^{(p,q)} = (N_{ij})$
where $N_{ij}$ is the number of all suitably connected $(p,q)$-mosaics
that have the $i$-th $l$-state and the $j$-th $r$-state in the set of $2^p$ states of the order  arranged above.

Furthermore, we split the state matrix $N^{(p,1)}$, only when $q=1$,
into two $2^p \times 2^p$ matrices, namely $X_p$ and $O_p$ as follows.
Each $(i,j)$-entry of $X_p$ (or $O_p$) indicates the number of all suitably connected $(p,1)$-mosaics
that have the $i$-th $l$-state and the $j$-th $r$-state, and additionally
whose bottom mosaic tiles are $\hat{b}$-cp (or $b$-cp, respectively).
Obviously $N^{(p,1)} = X_p + O_p$.
\vspace{3mm}

\noindent {\em Direct construction of the state matrix $N^{(1,1)}$.} \\
From eleven mosaic tiles in Figure \ref{fig3},
we get the following state matrices;
$$X_1 = \begin{bmatrix} 1 & 1 \\ 1 & 1  \end{bmatrix}, \ \ 
O_1 = \begin{bmatrix} 1 & 1 \\ 1 & 4  \end{bmatrix} \ \ \mbox{and} \ \
N^{(1,1)} = \begin{bmatrix} 2 & 2 \\ 2 & 5  \end{bmatrix}.$$
Each of four suitably connected $(1,1)$-mosaics on the first line in the figure
represents each entry 1 of $X_1$,
each of left three mosaics on the second line represents each entry 1 of $O_1$,
and the remaining four mosaics represent $(2,2)$-entry 4 of $O_1$.
Note that the sum of all entries of $N^{(1,1)}$ is the total number of elements of $\mathbb{S}^{(1,1)}$
which is obviously 11.

\begin{figure}[h]
\includegraphics{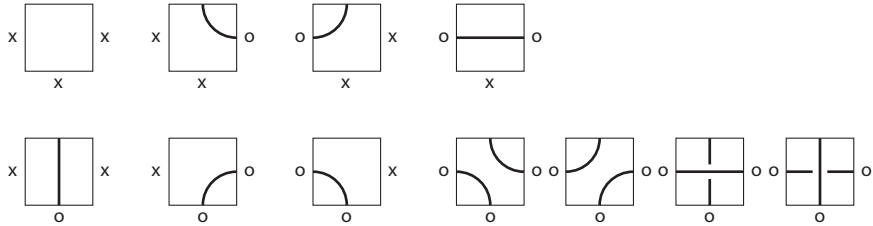}
\caption{Finding the state matrix $N^{(1,1)}$}
\label{fig3}
\end{figure}

Note that each element of $\mathbb{S}^{(1,1)}$ can be extended to exactly two knot 3-mosaics
because we have two choices of adjoining eight mosaic tiles surrounding it,
satisfying that all mosaic tiles are suitably connected.
This implies $D^{(3,3)} = 22$.

We can easily extend this argument to each element of $\mathbb{S}^{(m-2,n-2)}$
by adjoining $2m+2n-4$ proper mosaic tiles surrounding it.
Since each mosaic tile has even number of connection points,
a suitably connected $(m-2,n-2)$-mosaic has exactly even number of
connection points on its boundary.
To make a knot $(m,n)$-mosaic, all these connection points must be connected
pairwise via mutually disjoint arcs when we adjoin new mosaic tiles.
There are exactly two ways to do as illustrated in Figure \ref{fig4}.
Note that if it has no connection point on the boundary,
then we may add empty tiles or encircle the mosaic with a new circle.
\vspace{3mm}

\noindent {\bf Twofold rule.\/} A suitably connected $(m-2,n-2)$-mosaic can be extended
to exactly two different knot $(m,n)$-mosaics by augmenting the four sides with a row/column of tiles.

\begin{figure}[h]
\includegraphics{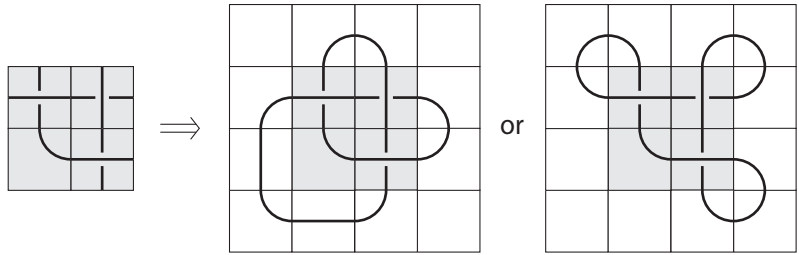}
\caption{The twofold rule}
\label{fig4}
\end{figure}

\section{State matrix $N^{(p,1)}$}
In this section, we establish the state matrix $N^{(p,1)}$ for $\mathbb{S}^{(p,1)}$.
For a $2^{k+1} \times 2^{k+1}$ matrix $N = (N_{ij})$,
the 11-quadrant (similarly 12-, 21- or 22-quadrant) of $N$ denotes
the $2^k \times 2^k$ submatrix $(N_{ij})$ where $1 \leq i, j \leq 2^k$
($1 \leq i \leq 2^k$ and $2^k+1 \leq j \leq 2^{k+1}$,
$2^k+1 \leq i \leq 2^{k+1}$ and $1 \leq j \leq 2^k$,
or $2^k+1 \leq i, j \leq 2^{k+1}$, respectively).

\begin{proposition}\label{prop:np1}
For the set $\mathbb{S}^{(p,1)}$ of all suitably connected $(p,1)$-mosaics,
the associated state matrix $N^{(p,1)}$ can be obtained as follows;
$$ N^{(p,1)} = X_p + O_p$$
where matrices $X_p$ and $O_p$ are defined by
$$ X_{k+1} = 
\begin{bmatrix} X_k & O_k \\ O_k & X_k  \end{bmatrix}
\ \mbox{and } \
O_{k+1} = 
\begin{bmatrix} O_k & X_k \\ X_k & 4 \, O_k  \end{bmatrix} $$
for $k=1, \dots, p-1$, starting with
$X_1 = \begin{bmatrix} 1 & 1 \\ 1 & 1  \end{bmatrix}$ and \
$O_1 = \begin{bmatrix} 1 & 1 \\ 1 & 4  \end{bmatrix}$.
\end{proposition}

\begin{proof}
The identity $N^{(p,1)} = X_p + O_p$ follows immediately from the definition of $X_p$ and $O_p$.
We will use the induction on $p$.
Matrices $X_1$ and $O_1$ are already found in the previous section.

Assume that matrices $X_k$ and $O_k$ satisfy the statement.
Let $S^{k+1,1} = (M_{i,1})$ be a suitably connected $(k+1,1)$-mosaic of $\mathbb{S}^{(k+1,1)}$.
Consider the bottom mosaic tile $M_{k+1,1}$.
If it is $\hat{l} \hat{r} \hat{b}$-cp, for example,
then $S^{k+1,1}$ should be counted in an entry of the 11-quadrant of $X_{k+1}$.
This is because of the backwardness of lexicographical order of $2^{k+1}$ states.
In this case, $M_{k+1,1}$ has unique choice $T_0$ of mosaic tiles because of Choice rule.
Let $S^{k,1}$ be the associated suitably connected $(k,1)$-mosaic
obtained from $S^{k+1,1}$ by ignoring $M_{k+1,1}$.
Then the bottom mosaic tile of $S^{k,1}$ must be $\hat{b}$-cp,
so the associated state matrix for all possible $S^{k,1}$ is $X_k$.
The eight figures in Figure \ref{fig5} and the table below explain all eight cases
according to the presence of connection points of $M_{k+1,1}$.
Notice that only when $M_{k+1,1}$ is $lrb$-cp,
it has four choices of mosaic tiles $T_7$, $T_8$, $T_9$ and $T_{10}$.
Thus the associated submatrix must be $4 \, O_k$ instead of $O_k$.
Now we complete the proof.
\end{proof}

\begin{figure}[h]
\includegraphics{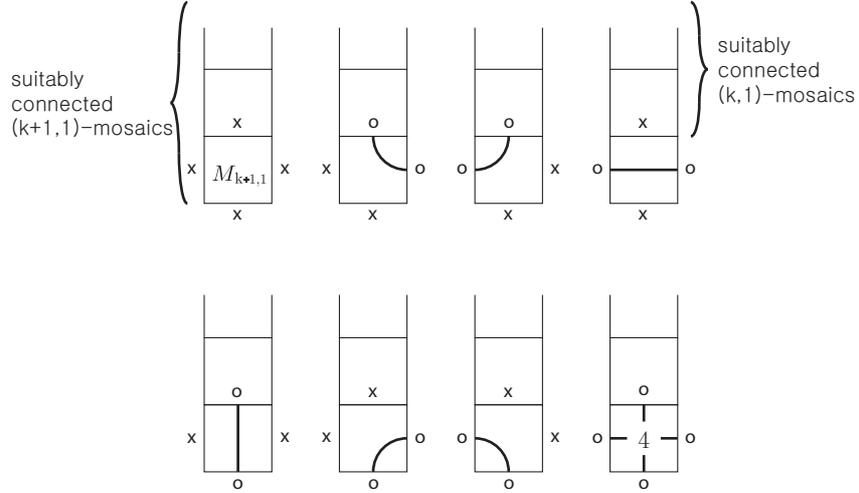}
\caption{Suitably connected $(k+1,1)$-mosaics}
\label{fig5}
\end{figure}

\begin{center}
\begin{tabular}{|c|c|c|c|c|}   \hline
 & {\em quadrants\/} & \multicolumn{2}{c|}{\em associated $M_{k+1,1}$\/} &
{\em submatrix\/} \\    \hline
\multirow{4}{8mm}{$X_{k+1}$}
 & 11-quadrant & $\hat{l} \hat{r} \hat{b}$-cp &
$T_0$ & $X_k$ \\
 & 12-quadrant & $\hat{l} r \hat{b}$-cp &
$T_3$ & $O_k$ \\
 & 21-quadrant & $l \hat{r} \hat{b}$-cp &
$T_4$ & $O_k$ \\
 & 22-quadrant & $l r \hat{b}$-cp &
$T_5$ & $X_k$ \\    \hline
\multirow{4}{8mm}{$O_{k+1}$}
 & 11-quadrant & $\hat{l} \hat{r} b$-cp &
$T_6$ & $O_k$ \\
 & 12-quadrant & $\hat{l} r b$-cp &
$T_2$ & $X_k$ \\
 & 21-quadrant & $l \hat{r} b$-cp &
$T_1$ & $X_k$ \\
 & 22-quadrant & $l r b$-cp &
$T_7$, $T_8$, $T_9$, $T_{10}$ & $4 \, O_k$ \\  \hline
\end{tabular}
\end{center}
\vspace{5mm}

For example, let us try to find $(10,11)$-entry of $O_4$.
This entry can be written as
$(2^3 + 0 \cdot 2^2 + 2^1 + 0 \cdot 2^0, 2^3 + 0 \cdot 2^2 + 2^1 + 2^0)$-entry,
and so counts the total number of all suitably connected $(4,1)$-mosaics
with xoxo $l$-state and ooxo $r$-state, and additionally
whose bottom mosaic tiles are $b$-cp as shown in Figure \ref{fig6}.
In this case, $M_{41}$ is $l r b$-cp, so has 4 choices of mosaic tiles.
Thus $M_{31}$ is $\hat{l} \hat{r} b$-cp, so it must be unique choice $T_6$.
Similarly $M_{21}$ and $M_{11}$ have 4 choices and unique choice, respectively.
Thus the entry is $4^2$.

\begin{figure}[h]
\includegraphics{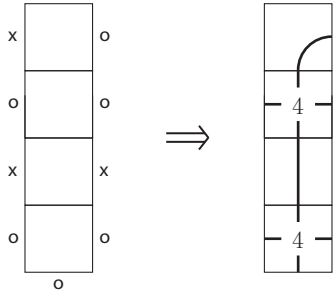}
\caption{Finding $(10,11)$-entry of $O_4$}
\label{fig6}
\end{figure}

\section{State matrix $N^{(p,q)}$ and the proof of Theorem \ref{thm:main}}

In this section, we find the state matrix $N^{(p,q)}$ for $\mathbb{S}^{(p,q)}$
and prove Theorem \ref{thm:main}.

\begin{proposition}\label{prop:npq}
For the set $\mathbb{S}^{(p,q)}$ of all suitably connected $(p,q)$-mosaics,
the associated state matrix $N^{(p,q)}$ is the following;
$$ N^{(p,q)} = (N^{(p,1)})^q.$$
\end{proposition}

\begin{proof}
We use the induction on $q$.
Assume that $N^{(p,k)} = (N^{(p,1)})^k$.
Let $S^{p,k+1}$ be a suitably connected $(p,k+1)$-mosaic in $\mathbb{S}^{(p,k+1)}$.
Also let $S^{p,k}$ and $S^{p,1}$ be the suitably connected $(p,k)$-mosaic obtained by
ignoring the rightmost column of $S^{p,k+1}$ and the suitably connected $(p,1)$-mosaic
which is just the rightmost column of $S^{p,k+1}$, respectively.
Then $r$-state of $S^{p,k}$ is the same as $l$-state of $S^{p,1}$ as shown in Figure \ref{fig7}.
Remark that $N^{(p,k+1)} = (N^{(k+1)}_{ij})$ is the state matrix for $\mathbb{S}^{(p,k+1)}$
where each entry $N^{(k+1)}_{ij}$ counts the number of all suitably connected $(p,k+1)$-mosaics
that have the $i$-th $l$-state and the $j$-th $r$-state in the set of $2^p$ states.
Also consider the state matrices $N^{(p,k)} = (N^{(k)}_{is})$
and $N^{(p,1)} = (N^{(1)}_{sj})$ defined similarly.
Among these suitably connected $(p,k+1)$-mosaics counted in each entry $N^{(k+1)}_{ij}$,
the number of all mosaics whose $r$-state of the $k$-th column
(or equally $l$-state of the $(k+1)$-th column) is the $s$-th state in the set of $2^p$ states
is the product of $N^{(k)}_{is}$ and $N^{(1)}_{sj}$.
Since all $2^p$ states can be appeared as states of connection points
where $S^{p,k}$ and $S^{p,1}$ meet,
we get
$$ N^{(k+1)}_{ij} = \sum^{2^p}_{s=1} N^{(k)}_{is} N^{(1)}_{sj}.$$
This implies that
$$ N^{(p,k+1)} = N^{(p,k)} N^{(p,1)} = (N^{(p,1)})^{k+1}.$$
\end{proof}

\begin{figure}[h]
\includegraphics{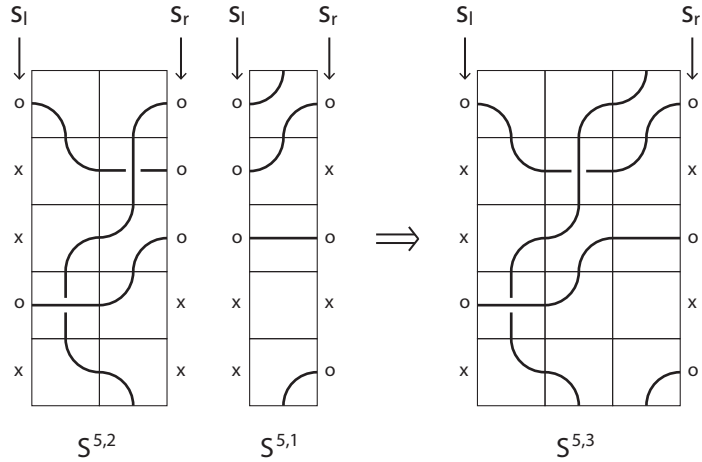}
\caption{Adjoining two suitably connected mosaics}
\label{fig7}
\end{figure}

Now we are ready to prove the main theorem.

\begin{proof}[Proof of Theorem \ref{thm:main}]
First assume that $m$ and $n$ are any integers at least 3.
Consider the set $\mathbb{S}^{(m-2,n-2)}$ of all suitably connected $(m-2,n-2)$-mosaics
and the associated state matrix $N^{(m-2,n-2)}$.
By the definition of the state matrix,
all rows represent all $2^{m-2}$ $l$-states
and all columns represent all $2^{m-2}$ $r$-states of mosaics of $\mathbb{S}^{(m-2,n-2)}$.
And each entry of the matrix counts the number of all suitably connected $(m-2,n-2)$-mosaics
having specific $l$-state and $r$-state.
Thus the total number of elements of $\mathbb{S}^{(m-2,n-2)}$ is the sum of all entries of the state matrix,
which is $\| N^{(m-2,n-2)} \|$.

Each suitably connected mosaic in $\mathbb{S}^{(m-2,n-2)}$ can be extended to
exactly two knot $(m,n)$-mosaics by Twofold rule.
Thus the total number of all knot $(m,n)$-mosaics $D^{(m,n)}$ is twice of $\| N^{(m-2,n-2)} \|$.
This fact combined with Proposition \ref{prop:np1} and \ref{prop:npq} completes the proof
except for the case that $m$ or $n$ is 2.

For the case of $m=2$,
we denote two $1 \times 1$ matrices $X_0 = \begin{bmatrix} 1  \end{bmatrix}$ and
$O_0 = \begin{bmatrix} 1  \end{bmatrix}$.
Then the same matrices $X_1$ and $O_1$ are obtained
from the recurrence relations in Proposition \ref{prop:np1},
and also $D^{(2,n)} = 2 \, \| (X_0+O_0)^{n-2} \| = 2^{n-1}$
which is already known.

For the case of $n=2$,
$ D^{(m,2)} = 2 \, \| (X_{m-2}+O_{m-2})^0 \| = 2^{m-1}$
because $(X_{m-2}+O_{m-2})^0$ is the $2^{m-2} \times 2^{m-2}$ identity matrix.
\end{proof}

\end{document}